\documentclass{article}

\usepackage{verbatim}
\usepackage{graphics}
\usepackage[pdftex]{graphicx}
\usepackage{amsmath, amssymb, amsfonts, amsthm}
\usepackage{url}
\usepackage{setspace}
\usepackage{algorithmic}
\usepackage{enumitem}
\usepackage{tabularx}
\setlist{nolistsep}

\def \PFq[#1]{\mathbf{P}^#1(\mathbb{F}_q)}
\def \PF[#1,#2]{PG(#1,#2)}

\newtheorem{theorem}{Theorem}

\newtheorem{lemma}[theorem]{Lemma}

\begin{document}

\title{Two theorems on point-flat incidences.}
\author{Ben Lund \footnote{Princeton University. Work on this project was supported by NSF grants DMS-1802787 and DMS-1344994.}}

\maketitle

\begin{abstract}
We improve the theorem of Beck giving a lower bound on the number of $k$-flats spanned by a set of points in real space, and improve the bound of Elekes and T\'oth on the number of incidences between points and $k$-flats in real space.
\end{abstract}

\section{Introduction}

Let $P$ be a set of $n$ points in $d$-dimensional real affine space.
For any integers $0 \leq k \leq d$ and $1 \leq r \leq n$, a $k$-flat ($k$-dimensional affine subspace) $\Gamma$ is $r$-rich if it contains at least $r$ points of $P$.
A $k$-flat $\Gamma$ is spanned by $P$ if $\Gamma$ contains $k+1$ points of $P$ that are not contained in a $(k-1)$-flat.
This paper gives new results on two well-studied questions:
\begin{enumerate}
\item How many $r$-rich $k$-flats can be determined by $P$?
\item How few $k$-flats can be spanned by $P$?
\end{enumerate}

A fundamental result in combinatorial geometry is the Szemer\'edi-Trotter theorem \cite{szemeredi1983extremal}, which gives an upper bound on the number of $r$-rich lines determined by a set of points.

Throughout this paper, $n$ is used for a positive integer, and $P$ is used for a set of $n$ points in real affine space.
None of the theorems depend on the dimension of the space.

\begin{theorem}[Szemer\'edi, Trotter]\label{th:sz-t}
For any integer $r > 1$, the number of $r$-rich lines determined by $P$ is bounded above by $O(n^2 r^{-3} + n r^{-1})$.
\end{theorem}

An alternate, and perhaps more well known, formulation of Theorem \ref{th:sz-t} is that the number of incidences between sets of $n$ points and $m$ lines in real space is bounded above by $O(n^{2/3}m^{2/3} + m + n)$.
These formulations are equivalent.

The first result of this paper is an upper bound on $r$-rich $k$-flats, for $k>1$.
In order to prove a nontrivial bound in this context, we need to place some restriction on the points or the flats.
To illustrate this point, let $L$ be a set of planes that each contain a fixed line, and $P$ a set of $n$ points contained in the same line.
Then, each plane of $L$ is $r$-rich for all $r \leq n$.

Several point-flat incidence bounds have been proved, using a variety of nondegeneracy conditions.
Initial work on point-plane incidences was by Edelsbrunner, Guibas, and Sharir \cite{edelsbrunner1990complexity}, who considered point sets with no three collinear points, and also incidences between planes and vertices of their arrangement.
Agarwal and Aronov \cite{agarwal1992counting} gave an asymptotically tight bound on the number of incidences between vertices and flats of an arrangement of $k$-flats; this bound was generalized by the author, Purdy, and Smith \cite{lund2011bichromatic} to incidences between vertices of an arrangement of flats and a subset of the flats of the arrangement.
Brass and Knauer \cite{brass2003counting}, as well as Apfelbaum and Sharir \cite{apfelbaum2007large}, showed that a point-flat incidence graph with many edges must contain a large complete bipartite graph.
Sharir and Solomon \cite{sharir2016incidences} obtain a stronger bound than that of Apfelbound and Sharir for point-plane incidences by adding the condition that the points are contained in an algebraic variety of bounded degree.

In this paper, we use the nondegeneracy assumption introduced by Elekes and T\'oth \cite{elekes2005incidences}.
For any real $\alpha \in (0,1)$, a $k$-flat $\Gamma$ is {\em $\alpha$-nondegenerate} if at most $\alpha |P \cap \Gamma|$ points of $P$ lie on any $(k-1)$-flat contained in $\Gamma$.
Using this definition, Elekes and T\'oth proved the following Szemer\'edi-Trotter type theorem for points and planes.
\begin{theorem}[Elekes, T\'oth]\label{th:et-2}
For any real $\alpha$ with $0 < \alpha < 1$ and integer $r > 2$, the number of $\alpha$-nondegenerate, $r$-rich planes determined by $P$ is bounded above by $O_\alpha(n^3r^{-4} + n^2 r^{-2})$.
\end{theorem}
The subscript in the $O$-notation indicates that the implied constant depends on those parameters listed in the subscript.

Elekes and T\'oth generalized Theorem \ref{th:et-2} to higher dimensions in the following, weaker form.
\begin{theorem}[Elekes, T\'oth]\label{th:et-k}
For each $k > 2$ there is a positive real constant $\beta_k$ such that, for any real $\alpha$ with $0 < \alpha < \beta_k$ and integer $r > k$, the number of $\alpha$-nondegenerate, $r$-rich $k$-flats determined by $P$ is bounded above by $O_{\alpha, k}(n^{k+1} r^{-k-2} + n^k r^{-k})$.
\end{theorem}

Elekes and T\'oth remarked that their argument can not be improved to replace the constants $\beta_k$ with $1$ for $k>2$ in Theorem \ref{th:et-k}.
Afshani, Berglin, van Duijn, and Nielsen \cite{afshani2016applications} used Theorem \ref{th:et-2} for an algorithm to determine the minimum number of $2$-dimensional planes needed to cover a set of points in $\mathbb{R}^3$, with a running time that depends on the required number of planes.
They mention that one of the obstacles to generalizing their algorithm to higher dimensions is the lack of a full generalization of Theorem \ref{th:et-2}.

The contribution of this paper is the following strong generalization of Theorem \ref{th:et-2}, which removes this limitation of Theoerem \ref{th:et-k}.

\begin{theorem}\label{th:et-gen}
For each integer $k \geq 1$, real $\alpha$ with $0 < \alpha < 1$, and integer $r > k$, the number of $\alpha$-nondegenerate, $r$-rich $k$-flats is bounded above by $O_{\alpha, k}(n^{k+1} r^{-k-2} + n^k r^{-k})$.
\end{theorem}
The case $k=1$ of this theorem is Theorem \ref{th:sz-t}, and the case $k=2$ is Theorem \ref{th:et-2}.
For $k>2$, it is new.

One well-known application of an incidence bound between points and lines is Beck's theorem \cite{beck1983lattice}.
Proving a conjecture of Erd\H{o}s \cite{erdos1975some}, Beck used a slightly weaker incidence bound than Theorem \ref{th:sz-t} to show that, if $P$ is a set of $n$ points such that no more than $s$ points of $P$ lie on any single line, then the number of lines spanned by $P$ is $\Omega(n(n-s))$.
In the same paper, Beck gave the following bound for flats of higher dimensions.
\begin{theorem}[Beck]\label{th:beck}
For each $k \geq 1$, there is are constants $c_k$ and $c_k'$ such that either $c_k n$ points of $P$ are contained in a single $k$-flat, or $P$ spans $c'_k n^{k+1}$ $k$-flats.
\end{theorem}

How large can the constant $c_k$ be in Theorem \ref{th:beck}?
For $k=1$, Beck showed that Theorem \ref{th:beck} holds for any $c_1 < 1$.
For $k=2$, if $P$ is a set of $n$ points of which $n/2$ lie on each of two skew lines, then $P$ spans $n$ planes, but no plane contains more than $n/2 + 1$ points of $P$.
Hence, Theorem \ref{th:beck} does not hold for $c_2 = 1/2$.
Beck's proof yields a constant $c_k$ of the form $c_k = e^{-ck}$ for some real $c>0$.
Do \cite{do2016extending} improved this by showing that Theorem \ref{th:beck} holds for any $c_k < 1/k$.

The second result of this paper is the following improvement to Theorem \ref{th:beck}.

\begin{theorem}\label{th:improvedBeck}
	For each integer $k \geq 1$, at least one of the following holds:
	\begin{enumerate}
		\item at least $(1-o(1))n$ points of $P$ are contained in a single $k$-flat, or
		\item at least $(\frac{1}{2}-o(1))n$ points of $P$ are contained in a single $(k-1)$-flat, or
		\item $k$ is odd and $(1-o(1))n $ points of $P$ are contained in the union of $k$ lines, or
		\item  $P$ spans $\Omega_k(n^{k+1})$ $k$-flats.
	\end{enumerate}
\end{theorem}

An immediate corollary of Theorem \ref{th:improvedBeck} is that, for any $c_k < 1/2$, there is a constant $c_k'$ such that Theorem \ref{th:beck} holds with for these choices of $c_k$ and $c_k'$.
Indeed, for odd $k$, any set of $(k+1)/2$ lines are contained in some $k$-flat.
Hence, if the third alternative in Theorem \ref{th:improvedBeck} holds, a simple averaging argument shows that some $k$-flat contains at least $(1-o(1)) ((k+1)/2) (n/k) > (1-o(1)) n/2$ points of $P$.

As noted above, we cannot take $c_2 = 1/2$ in Theorem \ref{th:beck}.
In fact, we can not take $c_k = 1/2$ for any $k>1$.
To see this, suppose that $P$ is contained in the union of a $(k-1)$-flat $\Gamma$ and a line $\ell$, with each of $\Gamma$ and $\ell$ containing $n/2$ points.
In this case, any $k$-flat spanned by $P$ contains either $\Gamma$ or $\ell$, and so $P$ spans at most $n/2 + \binom{n/2}{k-1} = O(n^{k-1})$ $k$-flats.

We remark that all of the new results in this paper hold for point sets in complex space, using the generalization of the Szemer\'edi-Trotter bound to complex space proved by T\'oth \cite{toth2015szemeredi} and Zahl \cite{zahl2012szemeredi}.

\section{Projective geometry and essential dimension}

In this section, we fix notation and review some basic facts of projective geometry, as well as results and definitions we need from \cite{lund2016essential}.
For convenience, we work in the $d$-dimensional real projective space $\mathbb{P}^d(\mathbb{R})$ instead of real affine space.
This does not affect any of the results.

The span of a set $X $ is the smallest flat that contains $X$, and is denoted $\overline{X}$.
We denote by $\overline{\Lambda,\Gamma}$ the span of $\Lambda \cup \Gamma$.
It is a basic fact of projective geometry that, for any flats $\Lambda, \Gamma$,
\begin{equation}\label{eqn:dimSpan} \dim(\overline{\Lambda,\Gamma})  + \dim(\Lambda \cap \Gamma) = \dim(\Lambda) + \dim(\Gamma).\end{equation}
More generally, using the fact that $\dim(\Lambda \cap \Gamma) \geq -1$, we have for any set $\mathcal{H}$ of flats that
\begin{equation}\label{eqn:dimSpanManyFlats}\dim \overline{\mathcal{H}} \leq |\mathcal{H}| - 1 +\sum_{\Lambda \in \mathcal{H}} \dim(\Lambda).\end{equation}

For any $k$-flat $\Lambda$ in $\mathbb{P}^d(\mathbb{R})$, the $(k+1)$-flats that contain $\Lambda$ correspond to the points of $\mathbb{P}^{d-k-1}(\mathbb{R})$.
The {\em projection from $\Lambda$} is the map $\pi_\Lambda: \mathbb{P}^d(\mathbb{R}) \setminus \Lambda \rightarrow \mathbb{P}^{d-k-1}(\mathbb{R})$ that sends a point $p$ to the $k+1$ flat $\overline{p,\Lambda}$.

Defined in \cite{lund2016essential}, the {\em essential dimension} $K=K(P)$ of $P$ is the minimum $t$ such that there exists a set $\mathcal{G}$ of flats such that
\begin{enumerate}
\item $P$ is contained in the union of the flats of $\mathcal{G}$,
\item each flat $\Lambda \in \mathcal{G}$ has dimension $\dim(\Lambda) \geq 1$, and
\item $\sum_{\Lambda \in \mathcal{G}} \dim(\Lambda) = t$.
\end{enumerate}

Denote by $f_k$ the number of $k$-flats spanned by $P$.
For each non-negative integer $i$, let $g_i$ be the largest number of points of $P$ contained in a set of essential dimension $i$.
For example, $g_1$ is the largest number of points contained in any single line, and $g_2$ is the largest number of points contained in any single plane, or the union of any pair of lines.
Note that $1=g_0 \leq g_1 \leq \ldots \leq g_K = n$.
A classical theorem of Beck \cite{beck1983lattice} is that $f_1 = n(n-g_1)$.
This was generalized to all dimensions in \cite{lund2016essential}, and this generalization is the main tool used here.
\begin{theorem}\label{thm:essentialDimBound}
For each $k$, there is a constant $c_k$ such that, if $n-g_k > c_k$, then
$$f_k = \Theta_k \left(\prod_{i=0}^k (n-g_i) \right).$$
If $n-g_k = 0$ ({i.e.} $k \geq K$), then
$$f_k = O_K\left(\prod_{i=0}^{2(K-1) - k} (n-g_i)\right),$$
and either $f_{k-1} = f_k = 0$ or $f_{k-1} > f_k$.
\end{theorem}

\section{Proof of Theorem \ref{th:et-gen}}

Recall from the introduction that a $k$-flat $\Lambda$ is {\em $\alpha$-nondegenerate} if at most $\alpha |P \cap \Lambda|$ points of $P$ lie on any $(k-1)$-flat contained in $\Lambda$.
We further say that $\Lambda$ is {\em essentially-$\alpha$-nondegenerate} if for each $P' \subset P \cap \Lambda$ such that the essential dimension of $P'$ is at most $k-1$, we have $|P'| \leq \alpha |P \cap \Lambda|$.
Note that an essentially-$\alpha$-nondegenerate flat is also $\alpha$-nondegenerate, but not necessarily the other way around.

The following bound on essentially-$\alpha$-nondegenerate flats was proved by Do \cite{do2016extending}.
\begin{theorem}[Do]\label{th:do}
For any integer $k \geq 1$, any real $\alpha$ with $0 < \alpha < 1$, and any integer $r > k$, the number of essentially-$\alpha$-nondegenerate, $r$-rich $k$-flats is bounded above by $O_{\alpha,k}(n^{k+1}r^{-k-2} + n^kr^{-k})$.
\end{theorem}

Theorem \ref{th:do} is also an immediate consequence of Theorem \ref{thm:essentialDimBound} together with the following theorem of Elekes and T\'oth \cite{elekes2005incidences}.
A $k$-flat $\Lambda$ is {\em $\gamma$-saturated} if $\Lambda \cap P$ spans at least $\gamma |\Lambda \cap P|^k$ different $(k-1)$-flats.
\begin{theorem}[Elekes, T\'oth]\label{th:elktot}
For any positive real $\gamma$, any integer $k \geq 1$, and any integer $r > k$, the number of $r$-rich $\gamma$-saturated $k$-flats is at most $O_{\gamma,k}(n^{k+1}r^{-k-2} + n^k r^{-k})$.
\end{theorem}

To prove Theorem \ref{th:do} from Theorem \ref{thm:essentialDimBound} and Theorem \ref{th:elktot}, observe that Theorem \ref{thm:essentialDimBound} implies that essentially-$\alpha$-nondegenerate $k$-flats are $\gamma$ saturated, for an appropriate choice of $\gamma$ depending on $\alpha$ and $k$.

In the remainder of this section, we deduce Theorem \ref{th:et-gen} from Theorem \ref{th:do}.

\subsection{The case $k=3$}
The case $k=3$ admits a simpler proof than the general theorem, which we give first.
The proof for arbitrary $k$ does not depend on this special case, but is built around a similar idea.

\begin{theorem}\label{th:et-3}
For any real $\alpha$ with $0 < \alpha < 1$ and integer $r > 3$, the number of $\alpha$-nondegenerate, $r$-rich $3$-flats is bounded above by $O_\alpha(n^{4}r^{-5} + n^3 r^{-3})$.
\end{theorem}
\begin{proof}
By Theorem \ref{th:do}, the number of essentially-$\alpha^{1/2}$-nondegenerate $r$-rich $3$-flats is bounded above by $O_\alpha(n^{4}r^{-5} + n^3 r^{-3})$.
If an $r$-rich $3$-flat $\Lambda$ is $\alpha$-nondegenerate but not essentially-$\alpha^{1/2}$-nondegenerate, then at least $\alpha^{1/2} |P \cap \Lambda| \geq \alpha^{1/2}r$ points of $P$ are contained in the union of two skew lines, neither of which contains more than $\alpha |P \cap \Lambda|$ points of $P$; hence, each of these lines contains at least $(\alpha^{1/2} - \alpha)r$ points.
By the Szemer\'edi-Trotter theorem (Theorem \ref{th:sz-t}), the maximum number of pairs of $((\alpha^{1/2} - \alpha)r)$-rich lines is bounded above by $O_\alpha(n^4 r^{-6} + n^2 r^{-2})$, which implies the conclusion of the theorem.
\end{proof}

\subsection{The general case}

The proof of Theorem \ref{th:et-3} given above does not generalize to higher dimensions, but the basic approach of bounding the number of $r$-rich $\alpha$-nondegenerate flats that are not also essentially-$\alpha'$-nondegenerate (for a suitable choice of $\alpha'$) does still work in higher dimensions.

It turns out that a distinguishing property of rich flats that are $\alpha$-nondegenerate but not essentially-$\alpha'$-nondegenerate is that they are \textit{special}, according to the following definition.
If a $k$-flat $\Lambda$ is $(r,\alpha)$-\textit{special}, then $\Lambda$ contains a set $\mathcal{G}$ of flats so that
\begin{enumerate}
	\item $\overline{\mathcal{G}} = \Lambda$,
	\item for each $G' \subseteq G$ with $|G'| > 1$, we have $\sum_{\Gamma \in \mathcal{G'}} \dim(\Gamma) < \dim(\overline{\mathcal{G'}})$, and
	\item each flat of $\mathcal{G}$ is $r$-rich and $\alpha$-nondegenerate.
\end{enumerate}

We first show that each rich flat that is $\alpha$-nondegenerate but not essentially-$\alpha'$-nondegenerate is special.

\begin{lemma}\label{th:specialFlats}
	Let $0 < \alpha < 1$, and let $r$ and $k$ be positive integers.
	If $\Lambda$ is an $r$-rich, $\alpha$-nondegenerate $k$-flat that is not also essentially-$\alpha'$-nondegenerate, then it is $(r',\alpha')$-special, for $\alpha' = (k + \alpha)(k + 1)^{-1}$ and $r'=(1-\alpha')|P \cap \Lambda|$.
\end{lemma}
\begin{proof}
	From the definition of essentially-$\alpha'$-nondegenerate, there is a collection $\mathcal{G}'$ of sub-flats of $\Lambda$ with $\sum_{\Gamma \in \mathcal{G}'} \dim(\Gamma) < k$ such that $|\bigcup_{\Gamma \in \mathcal{G}'} \Gamma \cap P| > \alpha' |P \cap \Lambda|$.
	We modify $\mathcal{G}'$ to obtain a set $\mathcal{G}$ satisfying the three conditions needed to show that $\Lambda$ is special, as follows.
	
	If $\Gamma \in \mathcal{G}'$ is not $\alpha'$-nondegenerate, then replace $\Gamma$ with the smallest flat $\Gamma' \subset \Gamma$ that contains at least $(\alpha')^{\dim(\Gamma) - \dim(\Gamma')}|P \cap \Gamma|$ points.
	Note that, since any flat $\Gamma'' \subset \Gamma'$ with $\dim(\Gamma'') = \dim(\Gamma') - 1$ contains fewer than $(\alpha')^{\dim(\Gamma) - \dim(\Gamma') + 1}|P \cap \Gamma| \leq \alpha'|P \cap \Gamma'|$ points of $P$, it follows that $\Gamma'$ is $\alpha'$-nondegenerate.
	Furthermore, the number of points in $\Gamma$ that are not also in $\Gamma'$ is at most $(1-(\alpha')^{\dim(\Gamma) - \dim(\Gamma')}) |P \cap \Gamma|$.
	
	Since $0 < \alpha' < 1$, we have for any integer $j \geq 1$ that 
	\[1-(\alpha')^j = \alpha'(1 - (\alpha')^{j-1}) + 1 - \alpha' \leq 1 - (\alpha')^{j-1} + 1 - \alpha'.\] 
	Hence by induction, $1 - (\alpha')^j \leq j(1-\alpha')$.
	Consequently, the number of points in $\Gamma$ that are not also in $\Gamma'$ is at most $ (\dim(\Gamma) - \dim(\Gamma'))(1-\alpha')|P \cap \Lambda|$.
	
	For any subset $\mathcal{G}'' \subset \mathcal{G}'$ such that $\sum_{\Gamma \in \mathcal{G''}} \dim(\Gamma) = \dim(\overline{\mathcal{G''}})$, replace $\mathcal{G}''$ with $\overline{\mathcal{G}''}$.
	Note that this does not increase $\sum_{\Gamma \in \mathcal{G'}} \dim(\Gamma)$.
	Repeat this procedure until the second condition in the definition of special holds.
	
	Remove from $\mathcal{G}'$ any flat that contains fewer than $(1-\alpha')|P \cap \Lambda|$ points to obtain the final set $\mathcal{G}$.
	Each remaining flat in $\mathcal{G}$ is $\alpha'$-nondegenerate and $(1-\alpha')|P \cap \Lambda|$-rich.
	The number of points that are contained in flats of $\mathcal{G}'$ but not in flats of $\mathcal{G}$ is at most $\sum_{\Gamma \in \mathcal{G}'} \dim(\Gamma) (1-\alpha')|P \cap \Lambda| < k (1-\alpha') |P \cap \Lambda|= (\alpha'-\alpha)|P \cap \Lambda|$ points.
	Hence, $|\bigcup_{\Gamma \in \mathcal{G}} \Gamma \cap P| > \alpha |P \cap \Lambda|$.
	If $\dim(\overline{\mathcal{G}}) < k$, then $\Lambda$ is not $\alpha$-nondegenerate, contrary to our assumption.
	Hence, $\dim(\overline{\mathcal{G}}) = k$, and $\Lambda$ is $(r',\alpha')$-special.
\end{proof}

Next, we show that the number of special flats is asymptotically smaller than the upper bound on the number of essentially-$\alpha$-nondegenerate flats coming from Theorem \ref{th:do}.

\begin{lemma}\label{th:boundingSpecial}
	For any $0 < \alpha < 1$ and positive integers $r,k$, the number of $(r,\alpha)$-special $k$-flats for $P$ is $O_{\alpha,k}(n^{k+1}r^{-k-2} + n^k r^{-k})$.
\end{lemma}
\begin{proof}
	We proceed by induction on $k$.
	The base case of $k=1$ is handled by Theorem \ref{th:sz-t}.
	Let $k>1$, and suppose that Theorem \ref{th:et-gen} has been shown to hold for all $k' < k$.
	
	Let $\mathcal{F}$ be the set of $(r,\alpha)$-special $k$-flats.
	We partition $\mathcal{F}$ into subsets $\mathcal{F}_b$ for each integer $1 \leq b < k$, and separately bound the size of each $\mathcal{F}_b$, as follows.
	
	First we assign each flat to one of the parts.
	For each $\Lambda \in \mathcal{F}$, let $\mathcal{G}_\Lambda$ be a minimal set of flats that shows that $\Lambda$ is special.
	Since $\mathcal{G}_\Lambda$ is minimal, we have that  $\overline{\mathcal{G}_\Lambda} = \Lambda$ but $\overline{\mathcal{G}_\Lambda \setminus \Gamma} \subsetneq \Lambda$ for each $\Gamma \in \mathcal{G}_\Lambda$.
	Let $\Gamma_\Lambda$ be an arbitrary flat in $\mathcal{G}_\Lambda$, let $b_\Lambda = \dim(\overline{\mathcal{G}_\Lambda \setminus \Gamma_{\Lambda}})$, and let $\mathcal{G}'_\Lambda = \mathcal{G}_\Lambda \setminus \Gamma_\Lambda$.
	Note that $\overline{\mathcal{G}'_\Lambda}$ is $(r,\alpha)$-special.
	Assign $\Lambda$ to $\mathcal{F}_{b_\Lambda}$.

	Now we fix $b$, and bound $|\mathcal{F}_b|$.
	The inductive hypothesis implies that
	$$|\{\overline{\mathcal{G}_\Lambda'} : \Lambda \in \mathcal{F}_b\}| = O_{\alpha,b}(n^{b+1}r^{-b-2} + n^b r^{-b}).$$
	Hence, it will suffice to bound the number of flats $\Lambda \in \mathcal{F}_b$ that share any fixed associated set $\mathcal{G}_\Lambda'$ by  $O_{\alpha,k}(n^{k-b}r^{-k+b})$.
	
	Let $\mathcal{R} \in \{\overline{\mathcal{G}_\Lambda'} : \Lambda \in \mathcal{F}_b\}$.
	Recall that $\pi_\mathcal{R}$ denotes the projection from $\mathcal{R}$.
	Let $M$ be a multiset of points in $\mathbb{P}^{d-1-b}$, with the multiplicity of a point $q$ equal to the number of points $p \in P$ so that $\pi(p) = q$.
	
	For each $\Lambda \in \mathcal{F}_b$ such that $\overline{\mathcal{G}_\Lambda'} = \mathcal{R}$, there is an $\alpha$-nondegenerate, $r$-rich flat $\Gamma$ such that $\overline{\Gamma, \mathcal{R}} = \Lambda$.
	Since $\overline{\Gamma,\mathcal{R}} = \Lambda$ and $\pi_{\mathcal{R}}(\Gamma)$ is disjoint from $\mathcal{R}$, we have that $\dim \pi_{\mathcal{R}}(\Gamma) = k-1-b$.
	We claim that $\pi_{\mathcal{R}}(\Gamma)$ is $(1-\alpha)r$-rich and $\alpha$-nondegenerate relative to $M$.
	First, note that $|\pi_\mathcal{R}(\Gamma) \cap M| = |\Gamma \cap P| - |\Gamma \cap \mathcal{R} \cap P|$.
	Since $\Gamma$ is $\alpha$-nondegenerate, $|\Gamma \cap \mathcal{R} \cap P| < \alpha |\Gamma \cap P|$, so $\pi_{\mathcal{R}}(\Gamma)$ is $(1-\alpha)r$-rich.
	Let $\Gamma'$ be a subflat of $\pi_{\mathcal{R}}(\Gamma)$, and let $\pi^{-1}(\Gamma') \subset \Gamma$ be the preimage of $\Gamma'$ in $\Gamma$.
	Note that $\dim \overline{\pi^{-1}(\Gamma'), \mathcal{R} \cap \Gamma} < \dim \Gamma$, hence $|\Gamma' \cap \pi_\mathcal{R}(P)| \leq \alpha |\Gamma \cap P| - |\Gamma \cap \mathcal{R} \cap P| \leq \alpha|\pi_\mathcal{R}(\Gamma) \cap \pi_\mathcal{R}(P)|$.
	Hence, $\pi_\mathcal{R}(\Gamma)$ is $\alpha$-nondegenerate.
	
	To complete the proof, we will use the following lemma, proved below.
	
	\begin{lemma}\label{th:mult-bound}
		Let $M$ be a multiset of points with total multiplicity $n$.
		The number of $r$-rich, $\alpha$-nondegenerate $k$-flats spanned by $M$ is bounded above by $(1-\alpha)^{-k}n^{k+1}r^{-k-1}$.
	\end{lemma}
	
	From Lemma \ref{th:mult-bound}, we get the required bound of $O(n^{k-b}r^{b-k})$ on the number of $(1-\alpha)r$-rich, $\alpha$-nondegenerate $(k-1-b)$-flats spanned by $\pi_\mathcal{R}(P)$, and this completes the proof of Theorem \ref{th:et-gen}.
\end{proof}

Here is the proof of Lemma \ref{th:mult-bound}.

\begin{proof}[Proof of Lemma \ref{th:mult-bound}]
There are $n^{k+1}$ ordered lists of $k+1$ points in $M$ (with repetitions allowed).
We show below that for any $r$-rich, $\alpha$-nondegenerate $k$-flat $\Lambda$, there are at least $(1-\alpha)^k r^{k+1}$ distinct lists of $k+1$ affinely independent points contained in $\Lambda$.
Since $k+1$ affinely independent points are contained in exactly one $k$-flat, this implies the conclusion of the lemma.

Let $\Lambda$ be an $r$-rich, $\alpha$-nondegenerate $k$-flat.
We will show that, for each $0 \leq k' \leq k$, $\Lambda$ contains $(1-\alpha)^{k'} r^{k'+1}$ distinct ordered lists of $k'+1$ affinely independent points of $M$.
We proceed by induction on $k'$.
The base case of $k'=0$ is immediate from the fact that $\Lambda$ is $r$-rich.
Let $k' > 0$, and suppose that $\Lambda$ contains $(1-\alpha)^{k' - 1} r^{k'}$ distinct ordered lists of $k'$ affinely independent points of $M$.

Let $V$ be the set of pairs $(\mathbf{v},p)$, where $\mathbf{v}$  is an ordered list of $k'$ affinely independent points of $M$ contained in $\Lambda$, and $p$ is a point of $M$ contained in $\Lambda \setminus \overline{\mathbf{v}}$.
By the inductive hypothesis, we know that there are $(1-\alpha)^{k'-1}r^{k'}$ choices for $\mathbf{v}$.
Since $\dim{\overline{\mathbf{v}}} = k'-1 \leq k-1$, the hypothesis that $\Lambda$ is $\alpha$-nondegenerate implies that $|\overline{\mathbf{v}} \cap M| \leq \alpha |\Lambda \cap M|$.
Consequently, for a fixed choice of $\mathbf{v}$, there are at least $(1 - \alpha) r$ choices for $p$.
Hence $|V| \geq (1-\alpha)^{k'}r^{k' + 1}$, as claimed.

\end{proof}

Now the proof of Theorem \ref{th:et-gen} is done.
Theorem \ref{th:do} gives the required bound on the number of essentially-$\alpha$-nondegenerate flats.
Lemma \ref{th:specialFlats} shows that the flats that are $\alpha$-nondegenerate but not essentially-$\alpha$-nondegenerate are special, and Lemma \ref{th:boundingSpecial} gives the required bound on special flats.

\section{Proof of Theorem \ref{th:improvedBeck}}

In this section, we show that Theorem \ref{th:improvedBeck} follows from Theorem \ref{thm:essentialDimBound}.

\begin{proof}
	Let $c_k$ be the constant in the lower bound of Theorem \ref{thm:essentialDimBound}.
	If $f_k \geq c_k n^{k+1}$, then the third alternative of Theorem \ref{th:improvedBeck} holds and we are done.
	
Suppose that $f_k < c_k n^{k+1}$.
Theorem \ref{thm:essentialDimBound} implies that there is a set $\mathcal{G}$ of flats such that $\sum_{\Gamma \in \mathcal{G}} \dim(\Gamma) \leq k$, at least $(1-o(1))n$ points of $P$ lie in some flat of $\mathcal{G}$, and each flat of $\mathcal{G}$ has dimension at least $1$.
If $\mathcal{G}$ consists of a single flat, then the first alternative holds and we're done.
Suppose that $|\mathcal{G}| > 1$.
We show below that, unless $k$ is odd and $\mathcal{G}$ is the union of $k$ lines, we can partition $\mathcal{G}$ into $\mathcal{G}_1, \mathcal{G}_2$ such that $\dim \overline{\mathcal{G}_1}, \dim \overline{\mathcal{G}_2} \leq k-1$.
Since either $|P \cap (\cup_{\Gamma \in \mathcal{G}_1} \Gamma)| \geq (1/2)|P \cap (\cup_{\Gamma \in \mathcal{G}} \Gamma)|$ or $|P \cap (\cup_{\Gamma \in \mathcal{G}_2} \Gamma)| \geq (1/2) |P \cap (\cup_{\Gamma \in \mathcal{G}} \Gamma)|$, this is enough to prove the theorem.

Let $\mathcal{G} = \{\Gamma_1, \Gamma_2, \ldots, \Gamma_m\}$, with $\dim(\Gamma_1) \leq \dim(\Gamma_2) \leq \ldots \leq \dim(\Gamma_m)$.
Let $\mathcal{G}_1 = \{\Gamma_1, \Gamma_2, \ldots, \Gamma_{m_1}\}$, with $m_1$ chosen as large as possible under the constraint $\dim \overline{\mathcal{G}_1} \leq k-1$.

Let $s = \dim \Gamma_{m_1 + 1}$.
Note that $\dim \overline{\mathcal{G}_1} + s \geq k-1$, since otherwise $\Gamma_{m_1 + 1}$ would be included in $\mathcal{G}$.
Since each flat in $\mathcal{G}_1$ has dimension at most $s$, by (\ref{eqn:dimSpanManyFlats}) we have
\begin{align*}
\sum_{\Gamma \in \mathcal{G}_1} \dim \Gamma &\geq \dim \overline{\mathcal {G}_1} + 1 - |\mathcal{G}_1|, \\
&\geq k - s - \frac{1}{s}\sum_{\Gamma \in \mathcal{G}_1} \dim \Gamma.
\end{align*}
Hence,
$$k - \sum_{\Gamma \in \mathcal{G}_2} \dim \Gamma \geq \sum_{\Gamma \in \mathcal{G}_1} \dim \Gamma \geq \frac{k-s}{1 + s^{-1}},$$
and so
$$\sum_{\Gamma \in \mathcal{G}_2} \dim \Gamma \leq \frac{s + ks^{-1}}{1 + s^{-1}}.$$

Since each flat in $\mathcal{G}_2$ has dimension at least $s$, we have $|\mathcal{G}_2| \leq \frac{1}{s} \sum_{\Gamma \in \mathcal{G}_2} \dim \Gamma$.
Applying (\ref{eqn:dimSpanManyFlats}), we have
\begin{align*}
\dim \overline{\mathcal{G}_2} &\leq |\mathcal{G}_2| - 1 + \sum_{\Gamma \in \mathcal{G}_2} \dim \Gamma, \\
&\leq (1 + s^{-1}) \sum_{\Gamma \in \mathcal{G}_2} \dim \Gamma - 1,\\
&\leq s + ks^{-1} - 1,\\
&\leq k,
\end{align*}
with equality only if $s=1$ and $|\mathcal{G}_2| = \frac{1}{s}\sum_{\Gamma \in \mathcal{G}_2} \dim \Gamma$; this occurs only if $\mathcal{G}$ is a set of lines.
Note that the case $s = k$ is eliminated by the assumption that $|\mathcal{G}| > 1$ together with the fact that $\sum_{\Gamma \in \mathcal{G}} \dim(\Gamma) \leq k$.

If $\mathcal{G}$ is a set of lines and $k$ is even, then $\mathcal{G}_2$ consists of at most $k/2$ lines, which span at most a $(k-1)$-flat.
Hence, if $\dim \overline{\mathcal{G}_2} = k$, then $\mathcal{G}$ must be a set of lines, and $k$ must be odd.
Since $\dim \overline{\mathcal{G}_1} \leq k-1$ by construction, this completes the proof.

\end{proof}

\bibliographystyle{plain}
\bibliography{purdy}

\end{document}